\documentclass[a4paper,11pt,reqno]{amsart}

\usepackage[top=3cm, bottom=3cm, left=3cm, right=3cm]{geometry}

\usepackage[utf8]{inputenx}

\usepackage[english]{babel}
\usepackage{microtype}

\usepackage{mathtools}
\usepackage{amssymb}
\usepackage{amsfonts}
\usepackage{amsthm}
\usepackage{mathrsfs}
\usepackage{stmaryrd}

\usepackage{graphicx}
\usepackage{tikz}
\usepackage{tikz-cd}
\usepackage{dynkin-diagrams}
\usepackage[all]{xy}

\usepackage{tensor}
\usepackage{centernot}
\usepackage{faktor}
\usepackage{xfrac}
\usepackage{esint}
\usepackage{bbm}

\usepackage{hyperref}

\usepackage{xcolor}
\usepackage{array, tabularx}
\usepackage{mathdots}
\usepackage{comment}
\usepackage{paralist}
\usepackage{xcolor}

\theoremstyle{plain}
\newtheorem{thm}{Theorem}[section]
\newtheorem{prop}[thm]{Proposition}
\newtheorem{lem}[thm]{Lemma}

\theoremstyle{definition}

\newtheorem{ex}[thm]{Example}

\theoremstyle{remark}
\newtheorem{rem}[thm]{Remark}

\begin{document}
\newpage
	
	\title{Compact lcK manifolds with Gauduchon parallel vector fields}

\author{Andrei Moroianu}

	\address{Andrei Moroianu\newline
		\textsc{ \indent Université Paris-Saclay, CNRS,  Laboratoire de mathématiques d'Orsay, \newline  \indent 91405, Orsay, France, and \newline  \indent Institute of Mathematics “Simion Stoilow” of the Romanian Academy, \newline  \indent 21 Calea Grivitei, 010702 Bucharest, Romania}}
\email{andrei.moroianu@math.cnrs.fr}

  \author{Giacomo Perri}

	\address{Giacomo Perri \newline
		\textsc{\indent Institut for Matematik, Aarhus University\newline 
			\indent 8000, Aarhus C, Denmark}}
	\email{g.perri@math.au.dk}

\thanks{The second author is supported by the Sapere Aude project $\lq\lq$Conformal geometry: metrics and cohomology" at Aarhus University}
\begin{abstract}
    We prove that on a compact locally conformally K\"ahler manifold, the existence of a non-trivial parallel vector field with respect to a Gauduchon connection $\nabla^t$, for $t\ne \frac{1}{3}$, forces the metric to be either Vaisman (for $t=-1$), globally conformally K\"ahler (for $t=1$), or K\"ahler for all other values of $t$. As a consequence, we characterize compact Vaisman manifolds in terms of their Bismut holonomy group. Moreover, we show that in complex dimension two, the existence of a non-trivial $\nabla^\frac{1}{3}$-parallel vector field forces the metric to be globally conformally K\"ahler.
\end{abstract}
\maketitle
\section{Introduction}

A central question in non-K\"ahler Hermitian geometry, posed by to S.T. Yau \cite[Problem 87]{Yau90}, asks whether one can say something nontrivial about a compact Hermitian manifold of complex dimension $n$ whose holonomy group with respect to some Hermitian connection is a proper subgroup of $\mathrm{U}(n)$.

Two natural Hermitian connections on a Hermitian manifold are the Bismut connection $\nabla^b$, characterized by the fact that its torsion is totally skew-symmetric, and the Chern connection $\nabla^c$, defined by the fact that its $(0,1)$-part is equal to the $\bar\partial$ operator.

More generally, a natural family of Hermitian connections for which one might try to investigate Yau's question was
introduced by Gauduchon in \cite{Gau97}. The Gauduchon connections
$\{\nabla^t\}_{t\in\mathbb{R}}$ form the affine line joining the Bismut
connection and the Chern connection, namely
\begin{equation}
\nabla^t
=
\frac{1-t}{2}\nabla^b
+
\frac{1+t}{2}\nabla^c.
\end{equation}
Thus, $\nabla^{-1}=\nabla^b$ and $\nabla^1=\nabla^c$. At $t=0$, one
recovers the first canonical connection, obtained by projecting the
Levi--Civita connection onto the holomorphic tangent bundle. The connection
$\nabla^{1/3}$ has torsion of minimal norm within the Gauduchon family and
is therefore called the ``minimal Gauduchon connection''. As we shall see, this parameter corresponds to an
exceptional case in our setting.
Although all of these connections coincide with the Levi--Civita connection in the K\"ahler case, they are generally distinct in the non-K\"ahler setting. 

Motivated by Yau's question, we investigate the existence of nonzero vector fields
that are parallel with respect to a Gauduchon connection on a compact
Hermitian manifold. Indeed, if $V$ is $\nabla^t$-parallel, then so is $JV$,
and the holonomy group $\text{Hol}(\nabla^t)$ is contained in $\mathrm{U}(n-1)$. 
In this paper, we study this problem in the setting of locally conformally K\"ahler geometry, a type of non-K\"ahler Hermitian geometry interesting from the conformal point of view. 

A Hermitian manifold $(M^{2n},J,g)$ is called locally conformally
K\"ahler (lcK) if its fundamental $2$-form $\omega$ satisfies
$$
d\omega=\theta\wedge\omega
$$
for some closed one-form $\theta$, called the Lee form of $\omega$. The vanishing, exactness, and Levi--Civita parallelism of $\theta$ correspond, respectively, to the K\"ahler, globally conformally K\"ahler, and Vaisman case. 
The Lee form therefore provides a natural measure of
the failure of an lcK metric to be K\"ahler.

In \cite{Mor15,MMP20} it is proved that a compact lcK manifold admitting a nonzero Levi--Civita parallel vector field is globally conformally K\"ahler or Vaisman. 
Consequently, a compact lcK manifold is
Vaisman if and only if its Levi--Civita holonomy group is contained in
$\mathrm{SO}(2n-1)$. 
It is then natural to investigate if an analogous rigidity result holds when the Levi--Civita connection is replaced by some Gauduchon connection.  
A partial result in this direction was obtained by Andrada and Villacampa in the Bismut case. They prove in \cite{AV22} that on a compact lcK manifold, the Lee form is Levi--Civita parallel if and only if it is Bismut parallel. This means that for compact Vaisman manifolds, the Bismut holonomy group is contained in $\mathrm{U}(n-1)$. We will show that this property actually characterizes compact Vaisman manifolds by proving the following result.

\begin{thm}\label{thm:holb-U(n-1)-vaisman}
Let $(M^{2n},J,\omega,\theta)$ be a compact lcK manifold.
Then the following are equivalent:
\begin{enumerate}
\item[\normalfont(i)] $(M,J,\omega,\theta)$ is Vaisman;
\item[\normalfont(ii)] $\mathrm{Hol}(\nabla^b)\subset \mathrm{U}(n-1)$.
\end{enumerate}
\end{thm}

More generally, we show that, for generic values of $t$, the existence of a nonzero $\nabla^t$-parallel vector
field forces the lcK metric to be K\"ahler, while for special parameters such as those associated with the Bismut and the Chern connections, a non-K\"ahler metric can occur. Our main result is the following.

\begin{thm}\label{maintheorem}
    Let $(M,J,\omega,\theta)$ be a compact lcK manifold of complex dimension $n\geq 2$ admitting a nontrivial $\nabla^t$-parallel vector field $V$, where $\nabla^t$ denotes the Gauduchon connection associated with the parameter $t\in\mathbb{R}$. Then the following hold:
    \begin{enumerate}
        \item[\normalfont(i)] If $t\ne \pm 1,\frac{1}{3}$, then $\theta\equiv 0$, so the metric $\omega$ is K\"ahler, and $V$ is parallel with respect to the Levi-Civita connection.
        \item[\normalfont(ii)] If $t=-1$, then the metric $\omega$ is Vaisman. Moreover, in complex dimension at least three, $V$ is a linear combination with constant coefficients of the Lee and anti-Lee vector fields $\theta^\sharp$ and $J\theta^\sharp$.
        \item[\normalfont(iii)] If $t=1$ then the Lee form $\theta$ is exact; in particular, $(M,J,\omega,\theta)$ is globally conformally K\"ahler. Moreover, there exist a complete simply connected K\"ahler manifold $(N,g_N,\omega_N)$ of real dimension $2n-2$, a smooth real function $c:\mathbb{R}^2\to\mathbb{R}$ (non-constant, if $\theta\not\equiv 0$), and a discrete co-compact group $\Gamma$ acting freely and totally discontinuously on $\mathbb{R}^2\times N$, preserving the metric $ds^2+dt^2+e^{c(s,t)}g_N$, the Hermitian $2$-form $ds\wedge dt+e^{c(s,t)}\omega_N$ and the vector fields $\partial_s$ and $\partial_t$, such that $M$ is diffeomorphic to $ (\mathbb{R}^2\times N)/\Gamma$, and the structure $(g,\omega,\theta)$ corresponds to $(ds^2+dt^2+e^{c(s,t)}g_N,ds\wedge dt+e^{c(s,t)}\omega_N,dc)$ through this diffeomorphism.
        \item[\normalfont(iv)] If $t=\frac{1}{3}$ then $V,JV$ are orthogonal to the Lee vector field $\theta^\sharp$. Moreover, in complex dimension $n=2$, the Lee form $\theta$ is exact; in particular, $(M,J,\omega,\theta)$ is globally conformally K\"ahler.
    \end{enumerate}
  \end{thm}

The proof of Theorem~\ref{maintheorem} relies on different arguments depending on the value of $t$.
We first decompose the Lee form as
$$
\theta=aV^\flat+bJV^\flat+\theta_0,
$$
where $\flat$ denotes metric duality, and $\theta_0$ is orthogonal to $V^\flat$ and $JV^\flat$.
In the lcK setting, we can express the Gauduchon connections in terms of the Levi--Civita connection and the Lee form $\theta$. This allows us to compute the exterior derivatives and codifferentials of $V^\flat$ and $JV^\flat$.
In the generic case $t\ne\pm 1,\frac{1}{3}$, we take contractions and wedge products of $dV^\flat,dJV^\flat$ and use identities such as $d^2V^\flat=d^2JV^\flat=0$ to obtain new relations. We then use compactness and integration arguments to deduce that $\theta\equiv0$.
In the Bismut case, we distinguish between the cases $\theta_0\not\equiv 0$ and $\theta_0\equiv 0$. We first assume that $\theta_0\not\equiv 0$ and note that $V$ and $JV$ are Killing vector fields. By taking Lie derivatives, we produce new Killing vector fields. Wedge-product arguments then show that this case can occur only in complex dimension two. When $\theta_0\equiv 0$, we prove that $a$ and $b$ are constant.
In the Chern case, we first prove that the Lee vector field is a combination of $V$ and $JV$. Once this is established, we show that the lifts $\tilde V$ and $\tilde{JV}$ to the universal cover are parallel with respect to the Levi--Civita connection of a metric conformal to the lift of $g$. This yields a splitting of the universal cover from which we will deduce the exactness of the Lie form $\theta$.

Finally, we treat the case $t=\frac{1}{3}$. Using integration arguments, we prove that $a=b=0$. In complex dimension two, using suitable commuting vector fields, we identify the universal cover $\tilde M$ with $S\times \mathbb R^2$, where $S$ is a complete connected surface. The existence of a nontrivial Killing vector field on $S$ with nonempty zero set reduces the possibilities to $S=\mathbb S^2,\mathbb R^2$. In each case, we deduce that the Lee form $\theta$ is exact.

While in complex dimension two we give a complete picture of the existence of a $\nabla^{\frac13}$-parallel vector field, also providing an example in which the metric is gcK but non-K\"ahler, the problem remains open in higher dimensions. 
However, in Remark~\ref{high_dim_ex} we construct examples in every complex dimension of compact gcK non-K\"ahler manifolds admitting a nontrivial $\nabla^{\frac13}$-parallel vector field.

We note that the techniques used for the other parameters do not
extend to the case $t=\frac{1}{3}$. The arguments in the Bismut and Chern
cases rely on geometric properties specific to the corresponding
values $t=-1$ and $t=1$. For $t\ne\pm1,\frac{1}{3}$, on the other hand, we derive
relations between $V$ and $\theta$ which can be manipulated to show that
$\theta\equiv 0$. At $t=\frac{1}{3}$, these relations yield less information,
since some terms are multiplied by the factor $(t-\frac{1}{3})$ and
therefore vanish.

We organize the paper as follows. We start with a preliminary section in which we recall basic notions of lcK geometry and Gauduchon connections. We also present some of the consequences of the existence of a nontrivial $\nabla^t$-parallel vector field on a compact lcK manifold, and conclude the section with some remarks on the Bismut connection in the lcK setting. In the following sections, we study the cases of Theorem~\ref{maintheorem} one by one, starting with 
the generic case $t\ne \pm1,\frac{1}{3}$,
then moving to the Bismut case, and to the Chern one. In the last section, we treat the exceptional case $t=\frac{1}{3}$.

\hfill

\noindent {\it Acknowledgments.}
The second author sincerely thanks Alexandra Otiman for her support, encouragement, and many valuable discussions and suggestions.

\section{Preliminaries}
In this section, we collect the preliminary material needed for the proof of Theorem \ref{maintheorem}. We begin by recalling the notion of locally conformally K\"ahler (lcK) manifolds together with several identities that will be used throughout the manuscript. We then introduce the family of Gauduchon connections and study their properties in the lcK setting. Assuming the existence of a nontrivial vector field parallel with respect to a Gauduchon connection, we derive several identities involving this vector field, including formulas for its codifferential, exterior derivative, and other related quantities that will be used repeatedly later. Finally, we discuss the Bismut connection in the lcK setting. Since the Bismut connection is a distinguished member of the family of Gauduchon connections, we conclude by establishing some additional properties that are specific to this case.

\subsection{LcK manifolds}
A locally conformally K\"ahler (lcK) manifold is a connected Hermitian manifold $(M,J,g,\omega,\theta)$ of complex dimension $n\geq 2$, where
$\omega:=g(J\cdot,\cdot)$
is the fundamental $2$-form associated with the Riemannian metric $g$, and $\theta$ is a closed $1$-form, called the \emph{Lee form} of $\omega$, satisfying
\begin{align}\label{lcKdefinition}
d\omega=\theta\wedge\omega.
\end{align}
If $\theta\equiv 0$, then $\omega$ is Kähler. An lcK metric is called \emph{proper} if it is not Kähler.
If $\theta=d\varphi$ is exact, then the conformally equivalent metric
$e^{-\varphi}\omega$
is K\"ahler. In this case, $\omega$ is called \emph{globally conformally K\"ahler} (gcK). The lcK structure is called \emph{strict} if the Lee form $\theta$ is not exact.
Finally, an lcK manifold $(M,J,g,\omega,\theta)$ is called \emph{Vaisman} if the Lee form $\theta$ is non-zero and parallel with respect to the Levi--Civita connection $\nabla^g$ of $g$.\\

We extend the action of $J$ to differential forms by setting
$$(J\alpha)(X_1,\ldots,X_k)
:=(-1)^k\alpha(JX_1,\ldots,JX_k),$$
where $\alpha$ is a $k$-form and $X_1,\ldots,X_k$ are vector fields. In particular, any $(1,1)$-form is $J$-invariant. Indeed, if $\beta$ is a $(p,q)$-form we have
$$J\beta = i^{\,q-p}\beta.$$

Throughout the remainder of this manuscript, we assume that $(M,J,g,\omega,\theta)$ is a compact lcK manifold, and we denote by $\nabla=\nabla^g$ the Levi--Civita connection of $g$.\\

Given a vector field $V$, we denote by  $V^\flat := g(V,\,\cdot\,) $ its metric-dual $1$-form. Conversely, for a $1$-form $\alpha$, we denote by $\alpha^\sharp$ its metric-dual vector field.

The covariant derivative of $J$ is given by \cite{Mor15}
\begin{align}\label{covariantderivativeofJ}
    \nabla_XJ=\frac{1}{2}\big(X\wedge J\theta+JX\wedge \theta\big),
\end{align}
where, for a vector field $X$ and a $1$-form $\alpha$, the endomorphism $X\wedge\alpha$ is defined by $(X\wedge\alpha )(Y):=X^\flat(Y)\alpha^\sharp-\alpha(Y)X$ for any vector field $Y$.

Recall that the codifferential of a $1$-form $\alpha$ is given by
\begin{align}\label{codifferential formula}
    \delta\alpha=-\sum_{j=1}^{2n}(\nabla_{e_j}\alpha)(e_j),
\end{align}
where $\{e_j\}_{j=1}^{2n}$ is a local orthonormal frame. As an immediate consequence of \eqref{codifferential formula}, if $f$ is a smooth function, then
\begin{align}\label{codifferentialfunctiononeform}
    \delta(f\alpha)=-\alpha^\sharp(f)+f\delta\alpha.
\end{align}
Throughout the manuscript, we will use the above identity \eqref{codifferentialfunctiononeform} without further reference. \\
We next prove that, if
$\{e_j\}_{j=1}^{2n}$
is a local orthonormal frame, then
\begin{align}\label{again}
    -\delta\theta+(1-n)|\theta|^2=\frac{1}{2}\sum_{j=1}^ndJ\theta(e_j,Je_j).
\end{align}
Using \eqref{covariantderivativeofJ}, we obtain
\begin{align}\label{nablaJbasis}
    \sum_{j=1}^{2n}\big(\nabla_{e_j}J\big)(e_j)=(n-1)J\theta^\sharp, \qquad \sum_{j=1}^{2n}\big(\nabla_{Je_j}J\big)(e_j)=-(n-1)\theta^\sharp.
\end{align}
Then \eqref{again} follows from
\begin{align*}
    \sum_{j=1}^{2n} dJ\theta(e_j,Je_j)&=\sum_{j=1}^{2n} e_j\big(J\theta(Je_j)\big)-Je_j\big(J\theta(e_j)\big)-J\theta\big(\nabla_{e_j}Je_j-\nabla_{Je_j}e_j\big)\\
&=\sum_{j=1}^{2n}e_j\big(\theta(e_j)\big)+Je_j\big(\theta(Je_j)\big)\\
&-J\theta\big((\nabla_{e_j}J)(e_j)+J\nabla_{e_j}e_j\big)+\theta\big((\nabla_{Je_j}J)(e_j)-\nabla_{Je_j}Je_j\big)\\
    &\overset{\eqref{codifferential formula},\eqref{nablaJbasis}}{=}-2\delta\theta+2(1-n)|\theta|^2.
\end{align*}

We conclude this section by recalling that the $2$-form $dJ\theta$ is of type $(1,1)$. Recall that the operator $d^c:=JdJ^{-1}$ satisfies
$$dd^c=2i\partial\overline{\partial}.$$
In particular, for any smooth function $f$, the $2$-form $dd^cf$ is of type $(1,1)$.
Since $\theta$ is closed, it is locally exact. Hence, there exists a smooth local function $\varphi$ such that $\theta=d\varphi$. Therefore,
\begin{align*}
    dJ\theta=dJd\varphi=dJdJ^{-1}\varphi=dd^c\varphi,
\end{align*}
which shows that $dJ\theta$ is a $(1,1)$-form.

\subsection{Gauduchon connections in the lcK setting}
Let $\{\nabla_g^t\}_{t\in\mathbb R}\equiv
\{\nabla^t\}_{t\in\mathbb R}$ denote the family of Gauduchon connections introduced in \cite{Gau97}:
\begin{align}\label{Gauduchongeneral}
    g(\nabla^t_XY, Z)=g(\nabla^g_XY, Z)-\frac{t-1}{4}d\omega(JX, JY, JZ)-\frac{t+1}{4}d\omega(JX, Y, Z).
\end{align}
For any $t$, the connection \(\nabla^t\) is Hermitian, that is,
\begin{align*}
    \nabla^t g=0,\qquad \nabla^t J=0.
\end{align*}
In particular, when $t=-1$, one recovers the Bismut connection $\nabla^b$ (see Section \ref{Bismut connection in the lcK setting}). The connection $\nabla^0$ is called the \emph{first canonical connection}; it is obtained by projecting the Levi--Civita connection onto the holomorphic tangent bundle. For $t=1$, the connection $\nabla^1$ coincides with the Chern connection, namely the unique Hermitian connection on the holomorphic tangent bundle that is compatible with the $\bar{\partial}$-operator.\\

In the lcK setting, formula \eqref{Gauduchongeneral} is equivalent to
\begin{equation}\label{nablatt}
\begin{split}
    \nabla^t_YZ=\nabla^g_YZ&-\frac{t}{2}\theta(JY)JZ+\frac{t-1}{4}\theta(JZ)JY\\
    &+\frac{t-1}{4}\omega(Y,Z)J\theta^{\sharp}-\frac{t+1}{4}\theta(Z)Y+\frac{t+1}{4}g(Y,Z)\theta^{\sharp}.
    \end{split}
\end{equation}
The torsion of $\nabla^t$ is therefore given by
\begin{equation}\label{torsion-t}
    \begin{split}
        T^t(Y,Z)=\frac{1-3t}{4}\big(\theta(JY)JZ-\theta(JZ)JY\big)+\frac{t-1}{2}\omega(Y,Z)J\theta^\sharp\\
-\frac{t+1}{4}\big(\theta(Z)Y-\theta(Y)Z\big).
    \end{split}
\end{equation}

\subsection{The Bismut connection in the lcK setting}\label{Bismut connection in the lcK setting}
The Bismut connection $\nabla^b$ of $g$ is the unique Hermitian connection whose torsion $3$-form is totally skew-symmetric. It coincides with the Gauduchon connection corresponding to the parameter $t=-1$. Hence, by \eqref{nablatt}, for any vector fields $X,Y$ on $M$, the Bismut connection can be written as
\begin{align}\label{bismutLC}
    \nabla^b_X Y
    =
    \nabla_X Y
    + \frac{1}{2}\Big(
    \theta(JX)JY
    - \theta(JY)JX
    - g(JX,Y)J\theta^\sharp
    \Big).
\end{align}
The torsion $3$-form of the Bismut connection is given by \cite[Lemma 3.1]{AV22}
$$
T^{-1}\equiv T^b=-J\theta\wedge\omega.
$$ 
The next lemma shows that, in the lcK setting, the torsion $3$-form is coclosed. As a consequence, the Bismut--Ricci form is $J$-invariant.
\begin{lem}\label{torsioncoclosed}
    On a lcK manifold, the torsion $3$-form of the Bismut connection is coclosed.
\end{lem}
\begin{proof}
   We first establish the identity
   \begin{align}\label{relation}
        *\big(J\theta\wedge \omega\big)=-\frac{\theta\wedge\omega^{n-2}}{(n-2)!}.
    \end{align}
    Consequently, $\delta T^b=-*d*(J\theta\wedge\omega)=0 $ since $d\theta=0$ and $d\omega=\theta\wedge \omega$.
    Let $\{e_j\}_{j=1}^{2n}$ be a local orthonormal frame. Then, for any $3$-form $\beta$, we have
\begin{align*}
    \beta\wedge *\big(J\theta\wedge \omega\big)&=\big(\beta,J\theta\wedge \omega\big)\frac{\omega^n}{n!}\\
    &=\frac{1}{2}\sum_{j=1}^{2n}\iota_{J\theta^\sharp}\beta (e_j,Je_j)\frac{\omega^n}{n!}\\
    &=\iota_{J\theta^\sharp}\beta\wedge\frac{\omega^{n-1}}{(n-1)!}\\
    &=\iota_{J\theta^\sharp}\bigg(\beta\wedge \frac{\omega^{n-1}}{(n-1)!}\bigg)+\beta\wedge\iota_{J\theta^\sharp}\frac{\omega^{n-1}}{(n-1)!}\\
    &=\beta\wedge \bigg(-\frac{\theta\wedge\omega^{n-2}}{(n-2)!}\bigg),
\end{align*}
which proves \eqref{relation}.
\end{proof}

Let $\text{R}^b$ be the curvature tensor of $\nabla^b$. The Bismut--Ricci tensor $\text{Ric}^b$ and the Bismut--Ricci form
$\rho^b$ of the Bismut connection $\nabla^b$ are defined, respectively,
by
$$
\text{Ric}^b(X,Y)
=\sum_{i=1}^{2n}\text{R}^b(e_i,X,Y,e_i),
\qquad
\rho^b(X,Y)
=\frac{1}{2}\sum_{i=1}^{2n}\text{R}^b(X,Y,e_i,Je_i),
$$
where $\{e_i\}_{i=1}^{2n}$ is a local orthonormal frame and $X,Y$ are
vector fields. 

Combining Lemma \ref{torsioncoclosed} with the following identities, which hold on every Hermitian manifold for any pair of vector fields $X,Y$ \cite[Corollary 3.2]{IP01},
\begin{equation}\label{relations}
    \begin{split}
    \text{Ric}^b(X,Y)-\text{Ric}^b(Y,X)&=-\delta T^b(X,Y),\\
    \text{Ric}^b(JX,JY)-\text{Ric}^b(Y,X)&=(n-1)\big(-\nabla^b_{JX}\theta(JY)+\nabla^b_Y\theta(X)\big),\\
    \rho^b(JX,JY)-\rho^b(X,Y)&=\delta T^b(JX,Y)
    -(n-1)\big(\nabla^b_X\theta(Y)-\nabla^b_Y\theta(X)\big),
    \end{split}
\end{equation}
we obtain the following proposition.
\begin{prop}
     Let $(M,J,g,\omega,\theta)$ be a lcK manifold. Then the following statements hold:
    \begin{enumerate}
        \item the Bismut--Ricci tensor $\text{Ric}^b$ is symmetric;
        \item the Bismut--Ricci form $\rho^b$ is $J$-invariant; equivalently, it is a $(1,1)$-form.
    \end{enumerate}
    Moreover, the Bismut--Ricci tensor $\text{Ric}^b$ is $J$-invariant if and only if the Lee form $\theta$ is holomorphic.
\end{prop}
\begin{proof}
   By Lemma~\ref{torsioncoclosed}, we have $\delta T^b=0$. It follows from
the first equation in \eqref{relations} that $\text{Ric}^b$ is symmetric.
A direct computation using~\eqref{bismutLC} shows that on an lcK manifold
we have
$$
\nabla^b\theta=\nabla\theta.
$$
Moreover, since $\theta$ is closed, $\nabla\theta$ is symmetric. We can
therefore rewrite the last two equations in \eqref{relations} as
\begin{equation}\label{relations2}
    \begin{split}
    \text{Ric}^b(JX,JY)-\text{Ric}^b(X,Y)
    &=(n-1)\bigl(-\nabla_{JX}\theta(JY)+\nabla_X\theta(Y)\bigr),\\
    \rho^b(JX,JY)-\rho^b(X,Y)
    &=0.
    \end{split}
\end{equation}
By \cite[Lemma~3]{MMO18}, the Lee form $\theta$ is holomorphic if and
only if $\nabla\theta$ is $J$-invariant. The result now follows
immediately from \eqref{relations2}.
\end{proof}
Finally, we recall that, on an lcK manifold, the Lee form is Levi--Civita parallel if and only if it is Bismut parallel:
\begin{thm}[{\cite[Theorem 3.6]{AV22}}]\label{BismutLCparallelAV}
    An lcK manifold is Vaisman if and only if $\nabla^b\theta=0$.
\end{thm}

\section{Gauduchon connections with parallel vector fields}

In this section, we fix a parameter $t\in\mathbb R$ and assume that $M$ admits a nontrivial vector field $V$ which is parallel with respect to the corresponding Gauduchon connection.
We establish here some results that will be used in what follows to prove Theorem~\ref{maintheorem} case by case.\\

Since every Gauduchon connection is Hermitian, we have
$$\nabla^tV=\nabla^tJV=0.$$
Moreover, since $\nabla^t$ is metric and $V$ is non-trivial, we may assume from now on that $V$ has unit length.

We decompose the Lee vector field $\theta^\sharp$ into its components along
$\langle V,JV\rangle$ and its orthogonal component:
\begin{equation}
\theta^\sharp=a\,V+b\,JV+\theta_0^\sharp,
\end{equation}
where
$$a:=\theta(V),
\qquad
b:=\theta(JV),$$
and $\theta_0^\sharp$ is orthogonal to both $V$ and $JV$.\\

Since $V$ is $\nabla^t$-parallel, its Levi--Civita covariant derivative can be written explicitly as 
\begin{equation}\label{parallel}
\begin{split}
    \nabla^g_YV=\frac{t}{2}\theta(JY)JV-\frac{t-1}{4}&bJY-\frac{t-1}{4}V^\flat(JY)J\theta^\sharp\\
    &+\frac{t+1}{4}aY-\frac{t+1}{4}V^\flat(Y)\theta^\sharp.
    \end{split}
    \end{equation}
Consequently, the Lie derivatives of the metric $g$ along $V$ and $JV$ are given by
\begin{equation}\label{tLiederivative}
\begin{split}
    \mathcal{L}_V g&=\frac{t+1}{4}\big(2ag-\theta\cdot V^\flat-J\theta\cdot JV^\flat\big),\\
    \mathcal{L}_{JV} g&=\frac{t+1}{4}\big(2bg-\theta\cdot JV^\flat+J\theta\cdot V^\flat\big),
    \end{split}
    \end{equation}
    where, for two $1$-forms $\alpha,\beta$, we set 
    $$\alpha\cdot\beta:=\alpha \otimes\beta+\beta\otimes \alpha.$$
Recall that, for a $1$-form $\alpha$, its exterior derivative is given by
$$d\alpha =\sum_{j=1}^{2n}e_j^\flat\wedge \nabla_{e_j}\alpha,$$
where $\{e_j\}_{j=1}^{2n}$ is a local orthonormal frame.
    Using \eqref{parallel}, we compute the exterior derivatives of $V^\flat$ and $JV^\flat$:
\begin{align}
    dV^\flat&=\frac{1-3t}{4}J\theta\wedge JV^\flat-\frac{t-1}{2}b\omega+\frac{t+1}{4}\theta\wedge V^\flat,\label{ddt1}\\
    dJV^\flat&=-\frac{1-3t}{4}J\theta\wedge V^\flat+\frac{t-1}{2}a\omega +\frac{t+1}{4}\theta\wedge JV^\flat.\label{ddt2}
\end{align}
Applying $d$ once again and using the identity $d^2=0$, we obtain
    \begin{align}
    0=d^2V^\flat&=\frac{1-3t}{4}dJ\theta\wedge JV^\flat-\frac{t-1}{2}\bigg(\frac{1-3t}{4}aJ\theta+db+\frac{3-t}{4}b\theta\bigg)\wedge \omega,\label{ddt3}\\
    0=d^2JV^\flat&=-\frac{1-3t}{4}dJ\theta\wedge V^\flat+\frac{t-1}{2}\bigg(-\frac{1-3t}{4}bJ\theta+da+\frac{3-t}{4}a\theta\bigg)\wedge \omega.\label{ddt4}
\end{align}
Since $V$ and $JV$ are $\nabla^t$-parallel, we have
\begin{align}\label{brackett}
    [V,JV]&=-T^t(V,JV)\overset{\ref{torsion-t}}{=}\frac{t-1}{2}\big(-bV+aJV-J\theta^\sharp\big)=-\frac{t-1}{2}J\theta_0^\sharp.
\end{align}
Consequently,
\begin{align}\label{dJthetatcase}
    dJ\theta(V,JV)=V(a)+JV(b)+\frac{t-1}{2}|\theta_0|^2.
\end{align}
Using \eqref{codifferential formula}, we compute the codifferentials of the $1$-forms $V^\flat$ and $JV^\flat$:
\begin{equation}\label{codifft}
    \begin{split}
        \delta V^\flat&=\frac{a}{2}(t+1)(1-n),\\
        \delta(JV^\flat)&=\frac{b}{2}(t+1)(1-n).
    \end{split}
\end{equation}
Using \eqref{covariantderivativeofJ} together with the identity
\begin{align*}
    \mathcal{L}_XJ(Y)&=\mathcal{L}_X JY-J\mathcal{L}_XY\\
    &=\nabla_XJ(Y)-\nabla_{JY}X+J\nabla_Y X,
\end{align*}
we compute the Lie derivative of $J$ along $V$:
    \begin{align}\label{LiederivativeofJ}
        \mathcal{L}_V J=\frac{1-t}{2}\big(JV\wedge \theta+V\wedge J\theta\big).
    \end{align}

The following lemmas will play a key role in the remainder of the paper.
\begin{lem}\label{theta0zeroonopenset}
    Suppose that $t\neq 1$, and let $U\subset M$ be an open subset on which
$\theta_0=0$. Then, on $U$, we have
    \begin{equation}\label{dalemmatheta0zero}
        da=-\frac{t+1}{2}a\theta, \qquad db=-\frac{t+1}{2}b\theta.
\end{equation}
    In particular, 
    \begin{equation}\label{V(a)theta0zero}
            V(a)=-\frac{t+1}{2}a^2, \qquad JV(b)=-\frac{t+1}{2}b^2.
    \end{equation}
    \end{lem}
\begin{proof}
On $U$, identities \eqref{ddt1},\eqref{ddt2},\eqref{ddt3},\eqref{ddt4} become
 \begin{equation}\label{dV-dJV2}
     \begin{split}
         dV^\flat&=\frac{1-t}{2}b(\omega-V^\flat\wedge JV^\flat),\\
         dJV^\flat&=-\frac{1-t}{2}a(\omega-V^\flat\wedge JV^\flat),\\
         0=d^2V^\flat&=\frac{1-t}{2}\left(db+\frac{t+1}{2}b\theta\right)\wedge (\omega-V^\flat \wedge JV^\flat),\\
         0=d^2JV^\flat&=-\frac{1-t}{2}\left(da+\frac{t+1}{2}a\theta\right)\wedge (\omega-V^\flat \wedge JV^\flat).
     \end{split}
 \end{equation}
Using that $d\theta=0$ and $\theta=aV^\flat+bJV^\flat$, we obtain
\begin{equation}
    da \wedge V^\flat+db \wedge JV^\flat+adV^\flat+bdJV^\flat=0.
\end{equation}
By \eqref{dV-dJV2}, the last two terms cancel, hence
\begin{equation}\label{da-db-V-JV}
    da \wedge V^\flat + db \wedge JV^\flat=0.
\end{equation}
The equation above shows that $da, db \in \langle V^\flat, JV^\flat \rangle$. Since $t\neq 1$, $da$, $db$ and $\theta$ are in $\langle V^\flat, JV^\flat \rangle$ and wedging with $\omega - V^\flat \wedge JV^\flat$ is injective on $\langle V^\flat, JV^\flat \rangle$, from (\ref{dV-dJV2}) we conclude
\begin{equation*}
    da=-\frac{t+1}{2}a\theta, \qquad db=-\frac{t+1}{2}b\theta
\end{equation*}
at each point of $U$.
 \end{proof}

The next lemma relates $da$, $db$, $\theta$, and $J\theta$, assuming that $t \neq 1$.
\begin{lem}
If $t\ne 1$, then 
\begin{align}
    Jda&=db+\frac{t+1}{2}\big(b\theta-aJ\theta\big).\label{integralsab1}
\end{align}
\end{lem}
\begin{proof}
    Applying $J$ to \eqref{ddt4}, and using that $dJ\theta$ is a $(1,1)$-form, we obtain
    \begin{align*}
            0=-\frac{1-3t}{4}dJ\theta\wedge JV^\flat+\frac{t-1}{2}\bigg(\frac{1-3t}{4}b\theta+Jda+\frac{3-t}{4}aJ\theta\bigg)\wedge \omega.
    \end{align*}
Adding this identity to
\eqref{ddt3}, we find
    \begin{align*}
        0&=\frac{t-1}{2}\bigg(\frac{t+1}{2}aJ\theta-\frac{t+1}{2}b\theta+Jda-db\bigg)\wedge \omega.
    \end{align*}
    Since $t\neq1$ and wedging a $1$-form with $\omega$ is injective, we obtain
\eqref{integralsab1}.
\end{proof}
We conclude this section by establishing the first relation among $V^\flat,a,b$ and $\theta$.
\begin{lem}\label{Lemmaforf}
    If $t\ne 1$ and $n>2$ then 
\begin{align}\label{fV}
    fV^\flat&=da-\frac{1-3t}{4}bJ\theta+\frac{3-t}{4}a\theta,
\end{align}
where $f$ is the smooth function defined by
\begin{equation}\label{fVprime}
    f:=V(a)+\frac{1-3t}{4}b^2+\frac{3-t}{4}a^2.
\end{equation}
\end{lem}
\begin{proof}
Since $t\ne 1$, by wedging \eqref{ddt4} with $V^\flat$ we obtain
$$0=\bigg(-\frac{1-3t}{4}bJ\theta+da+\frac{3-t}{4}a\theta\bigg)\wedge V^\flat\wedge \omega.$$
Since $n>2$, wedging with $\omega$ is injective on $2$-forms. Therefore,
$$0=\bigg(-\frac{1-3t}{4}bJ\theta+da+\frac{3-t}{4}a\theta\bigg)\wedge V^\flat.$$
Hence, there exists a smooth function $f$ such that
\begin{align*}
    fV^\flat&=da-\frac{1-3t}{4}bJ\theta+\frac{3-t}{4}a\theta.
\end{align*}
Applying this equality to $V$ gives \eqref{fVprime}.
\end{proof}

\section{The generic case \texorpdfstring{$t\neq \pm 1,\frac{1}{3}$}{t != +/- 1, 1/3}}

In this section, we assume that $t \neq \pm 1, \frac{1}{3}$ and prove  Theorem~\ref{maintheorem} under this assumption.
We distinguish two cases according to the complex dimension of $M$. We first treat the case $n>2$, and then the case $n=2$.\\

Assume that $n>2$. We begin by establishing, in Lemma \ref{Propositiontda}, a new relation involving $a$, $b$, $da$, $V^\flat$, and $\theta$, which is independent of the one obtained in Lemma \ref{Lemmaforf}. These two identities will be the key ingredients in proving that $a=b=0$. Once this is established, a straightforward, although nontrivial argument shows that $\theta\equiv 0$.

\begin{lem}\label{Propositiontda}
    Assume that $t\neq1$. Then there exists a smooth function $C\in C^\infty(M)$
such that
    \begin{align}\label{tda}
        tda=CV^\flat+\frac{(1-3t)(1+t)}{8}bJ\theta-\frac{(t+1)^2}{8}a\theta.
    \end{align}
\end{lem}
\begin{proof}
    Using Cartan's formula and the fact that $\theta$ is closed, we obtain
    $$\mathcal{L}_V\theta=d\iota_V\theta=d\big(\theta(V)\big)=da.$$
    Then, the Lie derivative of $J\theta$ along $V$ is given by
\begin{equation}\label{LieVJtheta section notpm1,1over3}
    \begin{split}
        \mathcal{L}_{V} J\theta&=\mathcal{L}_V J(\theta)+Jda\\
        &\overset{\ref{LiederivativeofJ}}{=}Jda+\frac{1-t}{2}\big(b\theta-|\theta|^2JV^\flat+aJ\theta\big).
    \end{split}
\end{equation}
    Applying Cartan's formula once again, we compute the contraction of $dJ\theta$ with $V$:
    \begin{equation}\label{contractionbyVofdJtheta}
        \begin{split}
            \iota_VdJ\theta&=\mathcal{L}_V J\theta-d\iota_VJ\theta\\
        &\overset{\ref{LieVJtheta section notpm1,1over3}}{=}Jda+\frac{1-t}{2}\big(b\theta-|\theta|^2JV^\flat+aJ\theta\big)+db.
        \end{split}
    \end{equation}
Contracting (\ref{ddt3}) with $V$ and using \eqref{contractionbyVofdJtheta}, we obtain
\begin{equation*}
    \begin{split}
        0&=\frac{3t-1}{4} JV^\flat\wedge \bigg(Jda+\frac{1-t}{2}\big(b\theta+aJ\theta\big)+db\bigg)\\
        &+\frac{t-1}{2}JV^\flat\wedge\bigg(-db+\frac{t-3}{4}b\theta+\frac{3t-1}{4}aJ\theta\bigg)\\
        &+\frac{t-1}{2}\omega\wedge\bigg(-V(b)-\frac{t+1}{2}ab\bigg).
    \end{split}
\end{equation*}
This expression immediately implies that the last term vanishes, and therefore
\begin{equation*}
    \begin{split}
        0&=\frac{3t-1}{4} JV^\flat\wedge \bigg(Jda+\frac{1-t}{2}\big(b\theta+aJ\theta\big)+db\bigg)\\
        &+\frac{t-1}{2}JV^\flat\wedge\bigg(-db+\frac{t-3}{4}b\theta+\frac{3t-1}{4}aJ\theta\bigg).
    \end{split}
\end{equation*}
Using \eqref{integralsab1} in the above equation, we deduce that
     $$0=JV^\flat\wedge \left(tJda+\frac{(1-3t)(1+t)}{8}b\theta+\frac{(t+1)^2}{8}aJ\theta\right).$$
Hence, there exists a smooth function $C$ such that
    \begin{align}\label{CJV}
        CJV^\flat=tJda+\frac{(1-3t)(1+t)}{8}b\theta+\frac{(t+1)^2}{8}aJ\theta.
    \end{align}
Applying $J$ to \eqref{CJV}, we obtain \eqref{tda}.
\end{proof}

Combining Lemmas \ref{Lemmaforf} and \ref{Propositiontda}, we obtain the following proposition.
\begin{prop}\label{a=b=00}
    If $n>2$ and $t\ne \pm 1,\frac{1}{3}$, then $a=b=0$.
\end{prop}
\begin{proof}
    Evaluating \eqref{fV} on $\theta_0^\sharp$, we obtain
    \begin{align}\label{firstthetazeroona}
        t\theta_0^\sharp(a)&=-\frac{(t+1)^2}{8}a|\theta_0|^2.
    \end{align}
Applying \eqref{tda} to $\theta_0^\sharp$ and multiplying the resulting identity by $t$, we find
\begin{equation}\label{secondthetazeroona}
    t\theta_0^\sharp(a)=\frac{t(t-3)}{4}a|\theta_0|^2.
\end{equation}
Since 
\begin{align}
    -\frac{(t+1)^2}{8}=\frac{t(t-3)}{4}
\end{align}
holds only for $t=1,\frac{1}{3}$, which are excluded by assumption, equations \eqref{firstthetazeroona} and \eqref{secondthetazeroona} imply that $$a\theta_0=0.$$ Applying the same argument to $b$, we obtain $$b\theta_0=0.$$
Let
$$T:=\{\theta_0\neq 0\}$$
be the open subset on which $\theta_0$ does not vanish, and let $U$ denote the interior of $M\setminus T$. Since $a\theta_0=b\theta_0=0$, we have $a=b=0$ on $T$. Therefore, \eqref{dalemmatheta0zero} and \eqref{V(a)theta0zero} hold on $T$. On the other hand, since $\theta_0=0$ on $U$, they also hold on $U$. As 
$$U \cup T$$
is dense in $M$, we conclude that \eqref{V(a)theta0zero}, namely
$$V(a)=-\frac{t+1}{2}a^2, \qquad JV(b)=-\frac{t+1}{2}b^2,$$
holds globally. Integrating the above identities over $M$ and using \eqref{codifft}, we obtain
$$\frac{(t+1)(1-n)}{2}\int a^2=-\frac{t+1}{2}\int a^2, \qquad \frac{(t+1)(1-n)}{2}\int b^2=-\frac{t+1}{2}\int b^2.$$
Therefore,
$$\frac{(t+1)(2-n)}{2}\int a^2=\frac{(t+1)(2-n)}{2}\int b^2=0.$$
Since $n>2$ and $t\ne -1$, we conclude that $a=b=0$.
\end{proof}
By Proposition \ref{a=b=00} we have $a=b=0$. Hence, \eqref{ddt3}, together with \eqref{ddt4} and \eqref{dJthetatcase}, implies that
\begin{align}\label{djthetatne1-113}
    dJ\theta=\frac{t-1}{2}|\theta_0|^2 V^\flat\wedge JV^\flat.
\end{align}
Differentiating \eqref{djthetatne1-113} and using \eqref{ddt1} and \eqref{ddt2}, together with the fact that $a=b=0$, we obtain
\begin{align*}
    0&=\frac{2}{t-1}d^2J\theta=d\big(|\theta|^2V^\flat\wedge JV^\flat\big)\\
    &=d\big(|\theta|^2\big)\wedge V^\flat\wedge JV^\flat+|\theta|^2\big(dV^\flat\wedge JV^\flat-V^\flat\wedge dJV^\flat\big)\\
    &\overset{\eqref{ddt1},\eqref{ddt2}}{=} \bigg(d\big(|\theta|^2\big)+\frac{t+1}{2}|\theta|^2\theta\bigg)\wedge V^\flat\wedge JV^\flat.
\end{align*}
Hence,
\begin{align}\label{maximumpoint}
    d\big(|\theta|^2\big)+\frac{t+1}{2}|\theta|^2\theta=0.
\end{align}
Since $M$ is compact, the function $|\theta|^2$ attains its maximum at some point $x\in M$. At this point, $d\big(|\theta|^2\big)(x) = 0$. As $t\neq -1$, equation \eqref{maximumpoint} then implies that $|\theta|^2(x)=0$. Since $x$ is a maximum point of the nonnegative function $|\theta|^2$, it follows that $|\theta|^2\equiv0$ on $M$. Therefore, $\theta\equiv 0$.\\

\vspace{3mm}

Assume now that $n=2$. We first prove that $\theta_0\equiv 0$. Wedging \eqref{ddt4} with $JV^\flat$ yields
\begin{align*}
    0&=-\frac{1-3t}{4}dJ\theta\wedge V^\flat\wedge JV^\flat+\frac{t-1}{2}\bigg(\frac{1-3t}{4}b^2+V(a)+\frac{3-t}{4}a^2\bigg)V^\flat\wedge JV^\flat\wedge \omega.
\end{align*}
     Therefore, multiplying the above equation by $\frac{4}{1-3t}$ we obtain
\begin{align*}
    0=\bigg(-dJ\theta+\bigg(\frac{t-1}{2}b^2+\frac{t-1}{1-3t}2V(a)+\frac{(t-1)(3-t)}{2(1-3t)}a^2\bigg)\omega \bigg)\wedge V^\flat\wedge JV^\flat.
\end{align*}
Combining this identity with \eqref{dJthetatcase}, we obtain that, if $V,JV,e,Je$ is a local orthonormal frame, then
\begin{equation*}
    \begin{split}
        dJ\theta(V,JV)+dJ\theta(e,Je)=&-\frac{t+1}{1-3t}V(a)+\frac{(t-1)(3-t)}{2(1-3t)}a^2\\
        &+JV(b)+\frac{t-1}{2}b^2+\frac{t-1}{2}|\theta_0|^2.
    \end{split}
\end{equation*}
Using  (\ref{again}), we find
\begin{align*}
    -\delta \theta&=-\frac{t+1}{1-3t}V(a)-\frac{(t+1)^2}{2(1-3t)}a^2+JV(b)+\frac{t+1}{2}b^2+\frac{t+1}{2}|\theta_0|^2\\
    &=-\frac{t+1}{1-3t}\bigg(V(a)+\frac{t+1}{2}a^2\bigg)+\bigg(JV(b)+\frac{t+1}{2}b^2\bigg)+\frac{t+1}{2}|\theta_0|^2.
\end{align*}
Integrating the above identity over $M$ and using \eqref{codifft}, we obtain $\int |\theta_0|^2=0$. Hence $\theta_0\equiv 0$. 
We claim that $\delta\theta=0$. Indeed, by Lemma \ref{theta0zeroonopenset} and (\ref{codifft}), 
$$\delta(aV^\flat)=-V(a)+a\delta V^\flat=0.$$
The same computation yields $$\delta(bJV^\flat)=0.$$
Therefore $\delta \theta=0$. It follows that
\begin{align*}
    0&=\int \big<\delta\theta,a^2\big>=\int 2a\theta^\sharp(a)\\
    &\overset{\ref{theta0zeroonopenset}}{=}-(t+1)\int a^2|\theta|^2=-(t+1)\int a^4+a^2b^2,
\end{align*}
and therefore $a=0$. The same argument shows that $b=0$, and hence $\theta\equiv 0$.

\section{The holonomy of the Bismut connection}

In this section, we prove Theorem \ref{maintheorem} for the case $t=-1$. Recall that the Bismut connection $\nabla^b$ is the Gauduchon connection corresponding to the parameter $t=-1$.\\

Our proof is divided into two cases. First, we assume that $\theta_0\not\equiv 0$. The main ingredient is the fact that the Lie bracket of Killing vector fields is again a Killing vector field. Starting from the Killing vector fields $V$ and $JV$ (which are Killing by \eqref{tLiederivative}), we repeatedly take Lie brackets to construct new Killing vector fields.
We then use the fact that the non-vanishing set of a nontrivial Killing vector field is dense on a connected Riemannian manifold. This will imply that the functions $a$, $b$, and $|\theta_0|$ are constant. We will show that this case cannot occur in complex dimension greater than two. In complex dimension two, however, it implies that the metric is Vaisman.
The second case is $\theta_0\equiv 0$. In this situation, we again prove that $a$ and $b$ are constant, from which it follows that the manifold is Vaisman.\\

\textbf{Case $\theta_0\not\equiv 0$. } 
Since $V$ and $JV$ are Bismut-parallel, they are Killing vector fields by \eqref{tLiederivative}. We recall that the Lie derivative along a Killing vector field preserves metric duality. Consequently, Cartan's formula can be applied interchangeably to vector fields and to their metric dual $1$-forms along Killing vector fields.
We will also use the fact that, on a connected Riemannian manifold, the set of points where a nontrivial Killing vector field does not vanish is dense \cite{Kob58}. Moreover, the Lie bracket of two Killing vector fields is again a Killing vector field. In particular, $J\theta_0^\sharp$ is a Killing vector field. Indeed, \eqref{brackett} immediately yields
\begin{equation}\label{[V,JV]}
        \mathcal{L}_V JV=J\theta_0^\sharp.
\end{equation}
Furthermore,
\begin{equation}\label{A}
    \begin{split}
        \mathcal{L}_V J\theta_0^\sharp&=-\mathcal{L}_{J\theta_0^\sharp} V\\
        &=-\big(\mathcal{L}_{J\theta_0^\sharp} V^\flat\big)^\sharp\\
    &=-\big(d\iota_{J\theta_0^\sharp} V^\flat+\iota_{J\theta_0^\sharp}dV^\flat\big)^\sharp\\
    &=-\big(\iota_{J\theta_0^\sharp}dV^\flat\big)^\sharp\\
    &\overset{\eqref{ddt1}}{=}-\big(\iota_{J\theta_0^\sharp}(J\theta\wedge JV^\flat+b\omega)\big)^\sharp\\
    &=-|\theta_0|^2JV +b\theta_0^\sharp,
    \end{split}
\end{equation}
and similarly,
    \begin{align}\label{B}
        \mathcal{L}_{JV} J\theta_0^\sharp&=|\theta_0|^2V-a\theta_0^\sharp
    \end{align}
are Killing vector fields.
Taking the Lie derivative once more along $V$ and along $JV$ of the above Killing vector fields, we obtain new Killing vector fields:
\begin{equation}\label{1}
    \begin{split}
        \mathcal{L}_V \big(b\theta_0^\sharp-|\theta_0|^2JV\big)&=-\mathcal{L}_{b\theta_0^\sharp-|\theta_0|^2JV} V\\
    &\overset{\eqref{ddt1}}{=}-\big(\iota_{b\theta_0^\sharp-|\theta_0|^2JV}(J\theta\wedge JV^\flat+b\omega)\big)^\sharp\\
    &=-\big(-a|\theta_0|^2JV+|\theta_0|^2J\theta^\sharp+b^2J\theta_0^\sharp+b|\theta_0|^2V\big)\\
    &=-\big(|\theta_0|^2+b^2\big)J\theta_0^\sharp,
    \end{split}
\end{equation}
where in the last equality we used the identity $J\theta^\sharp=aJV-bV+J\theta_0^\sharp$,
and similarly,
    \begin{align}
    \mathcal{L}_V\big(-a\theta_0^\sharp+|\theta_0|^2V\big)&=-\big(d|\theta_0|^2\big)^\sharp+abJ\theta_0^\sharp,\label{2}\\
    \mathcal{L}_{JV}\big(b\theta_0^\sharp-|\theta_0|^2JV\big)&=\big(d|\theta_0|^2\big)^\sharp+abJ\theta_0^\sharp,\label{3}\\
    \mathcal{L}_{JV}\big(-a\theta_0^\sharp+|\theta_0|^2V\big)&=-\big(|\theta_0|^2+a^2\big)J\theta_0^\sharp.\label{4}
\end{align}
We will use the following lemma.
\begin{lem}\label{Killing vector fields}
    Let $X$ and $Y$ be nontrivial Killing vector fields on a connected Riemannian manifold. If
$$X^\flat\wedge Y^\flat=0,$$
then there exists a constant $c\in\mathbb{R}$ such that $X=cY$.
\end{lem}
\begin{proof}
    Since $X^\flat\wedge Y^\flat=0$, there exists a smooth function $f$ defined on the dense open set $M_0:=\{Y\ne 0\}$ such that $X^\flat=fY^\flat$. Since both $X$ and $Y$ are Killing, we have
\begin{align*}
0&=\mathcal{L}_{fY}g\\
&=f\mathcal{L}_Yg + df\otimes Y^\flat + Y^\flat\otimes df\\
&= df\otimes Y^\flat + Y^\flat\otimes df.
\end{align*}
Let $W$ be orthogonal to $Y$. Evaluating the above identity on $(Y,Y)$ and $(W,Y)$ we obtain
\begin{align*}
    2df(Y)|Y|^2=0,
\end{align*}
and 
\begin{align*}
    df(W)|Y|^2=0.
\end{align*}
Since $M_0$ is dense, we conclude that 
$$df(Y)=df(W)=0.$$ 
Hence $df=0$ on $M_0$, and therefore $f$ is constant on $M$.
\end{proof}
Applying the above lemma with $Y=J\theta_0^\sharp$ and $X$ given by \eqref{1}, \eqref{4} and $\eqref{2}+\eqref{3}$, we deduce that the functions
$$|\theta_0|^2+b^2,\qquad
|\theta_0|^2+a^2,\qquad
ab$$
are constant. Indeed, if $c$ and $e$ are constants such that
\[\begin{cases}
ab=c\\
a^2-b^2=e,
\end{cases}\]
then 
\begin{align*}
    c^2&=(ab)^2=(b^2+e)b^2,
\end{align*}
that is,
\begin{align*}
    0&=b^4+b^2e-c^2.
\end{align*}
Hence, $b$ is constant, and consequently so $a$ and $|\theta_0|$. In particular, $|\theta|$ is constant.

We now prove that this case cannot occur in complex dimension greater than two. Assume $n>2$. By Lemma \ref{Lemmaforf}, there exists a smooth function
\begin{equation}
    f=a^2+b^2
\end{equation}
such that
\begin{equation}\label{fg}
\begin{split}
fV^\flat&=-bJ\theta+a\theta\\
&=(a^2+b^2)V^\flat+a\theta_0-bJ\theta_0.
\end{split}
\end{equation}
Hence,
$$a\theta_0-bJ\theta_0=0,$$
which implies that $a=b=0$, since $a$ and $b$ are constant and, by assumption, $\theta_0\not\equiv 0$.
From \eqref{ddt3} and \eqref{ddt4} we obtain
\begin{align*}
    0&=dJ\theta\wedge V^\flat=dJ\theta\wedge JV^\flat,
\end{align*}
and therefore, using \eqref{dJthetatcase}
\begin{align}\label{dJthtetaequalVJV}
    dJ\theta=-|\theta|^2 V^\flat\wedge JV^\flat.
\end{align}
Let $\{e_j\}_{j=1}^{2n}$ be a local orthonormal frame with $e_1=V$ and $e_2=JV$. 
Then identity \eqref{again} reduces to
\begin{align*}
    -\delta\theta+(1-n)|\theta|^2=-|\theta|^2.
\end{align*}
Integrating over $M$, we obtain
$$(n-2)\int |\theta|^2=0.$$
Since $n>2$, it follows that $\theta\equiv0$, a contradiction.\\

In complex dimension two, we argue differently. From \eqref{A} and \eqref{B}, we deduce that both $a\theta_0^\sharp$ and $b\theta_0^\sharp$ are Killing vector fields. If at least one of $a$ and $b$ is nonzero, then $\theta_0^\sharp$ is Killing, and therefore $\theta^\sharp$ is Killing as well. Since $\theta$ is also closed, it follows that it is Levi--Civita parallel. Therefore, the lcK metric $\omega$ is Vaisman. Assume now that $a=b=0$. Then, $J\theta^\sharp$ is Killing and, since the metric is lcK,
\[
\mathcal{L}_{J\theta^\sharp}\omega=0.
\]
These two properties imply that $J\theta^\sharp$ is holomorphic, and consequently so is $\theta^\sharp$. Since $\theta$ has constant length, \cite[Theorem 1]{MMO18} implies that $(M,J,\omega)$ is Vaisman.\\

\textbf{Case $\theta_0\equiv 0$.} Since $t=-1$, Lemma \ref{theta0zeroonopenset} immediately implies that $a$ and $b$ are constant. Since $V$ and $JV$ are $\nabla^b$-parallel, it follows that
$\nabla^b\theta=0$. On an lcK manifold, this condition is equivalent to $\nabla^g\theta=0$ (Theorem \ref{BismutLCparallelAV}). Therefore, $(M,J,\omega)$ is Vaisman.

\begin{rem} 
    We note that, in complex dimension $n>2$, the vector field $V$ is a linear combination of the Lee vector field $\theta^\sharp$ and the anti-Lee vector field $J\theta^\sharp$ with constant coefficients.
In complex dimension two, both cases appearing in the proof may occur. The first corresponds to $V$ being a linear combination of $\theta^\sharp$ and $J\theta^\sharp$ with constant coefficients, while the metric is not necessarily Bismut--Ricci flat; this is the case of Vaisman solvmanifolds \cite[Corollary~4.9]{AV22}.
The second case yields Bismut--Ricci flat surfaces. Indeed, up to a  normalization, there exists a $J$-orthonormal parallel Bismut basis $\{\theta^\sharp, J\theta^\sharp, V, JV\}$ and in complex dimension two the identity
\begin{align*}
    \rho^b(X,Y)&=\text{Ric}^b (X,JY)
\end{align*} 
holds for any vector fields $X,Y$ \cite[Proposition~3.1]{IP01}. This is the case of Hopf Vaisman surfaces \cite[Corollary 5.11]{AV22}.
\end{rem}

As an immediate consequence of Theorem \ref{maintheorem}, we obtain Theorem \ref{thm:holb-U(n-1)-vaisman}.
\begin{proof}[Proof of Theorem \ref{thm:holb-U(n-1)-vaisman}]
    The implication \emph{(i)$\Rightarrow$(ii)} is standard. On a Vaisman manifold the Lee form
is $\nabla^b$-parallel (Theorem \ref{BismutLCparallelAV}), and therefore so are $\theta^\sharp$ and $J\theta^\sharp$; consequently,
the Bismut holonomy reduces to $U(n-1)$. 
Conversely, assume that $\text{Hol}(\nabla^b)\subset U(n-1)$. Since we are considering the full
holonomy group, there exists a global nowhere-vanishing $\nabla^b$-parallel vector
field $V$. Theorem \ref{maintheorem} implies that $(M,J,\omega,\theta)$ is Vaisman.
\end{proof}
 
We conclude this section with a remark on the Bismut connection in the locally conformally balanced setting. We show that the analogue of Theorem~\ref{thm:holb-U(n-1)-vaisman} fails in this more general framework. This suggests that the lcK setting is the largest natural non-K\"ahler class in which the existence of a nontrivial Bismut-parallel vector field forces the Lee vector field to be parallel.
\begin{rem}
    Let $(X,J)$ be a complex manifold of complex dimension $n$ and let $\omega$ be a Hermitian metric on $X$. The metric $\omega$ is called \emph{balanced} if $d\omega^{n-1}=0$. In complex dimension two, any balanced metric is K\"ahler. The metric $\omega$ is called \emph{locally conformally balanced} (lcb) if
    $$d\omega^{n-1}=(n-1)\theta\wedge \omega^{n-1},\qquad d\theta=0,$$
    where $\theta$ denotes the Lee form of $\omega$. 
    By definition, in complex dimension two the notions of lcb and lcK coincide. On the other hand, in complex dimension greater than two, any lcK metric is lcb, whereas the converse does not hold.
    
    We observe that Theorem \ref{BismutLCparallelAV} remains valid in the lcb setting, provided that the Lee vector field is holomorphic and the anti-Lee vector field is Killing. We emphasize that, even in the lcK case, these additional assumptions on the Lee vector field do not imply that it is Levi--Civita parallel; see \cite[Lemma 2 and Section 3]{MMO18}.
    Let $(M,g,\omega,\theta)$ be a compact lcb manifold whose Lee vector field is holomorphic and such that $J\theta^\sharp$ is Killing. In this case, a straightforward computation shows that the Lie derivative of $\omega$ along $J\theta^\sharp$ vanishes. Therefore, by Cartan's formula, we obtain
$$\iota_{J\theta^\sharp}d\omega=\mathcal{L}_{J\theta^\sharp}\omega-d\iota_{J\theta^\sharp}\omega=0.$$
Following the same argument as in \cite[Corollary 3.2 and Theorem 3.6]{AV22}, it follows that $\theta$ is Bismut parallel if and only if it is Levi--Civita parallel. 

We now show that, in the more general lcb setting, the existence of a nontrivial Bismut-parallel vector field does not imply that the Lee form $\theta$ is Levi--Civita parallel, even under the additional assumptions above. Consider, for example, the compact Hermitian manifold 
$$(X,g,\omega):=(T\times M,g_T+g_0,\omega_T+\omega_0)$$
where $(T,g_T,\omega_T)$ is a flat complex torus and $(M,g_0,\omega_0,\theta_0)$ is a compact lcK manifold of complex dimension $n-1\ge 2$. By construction, $(X,\omega,\theta)$, where 
$$\theta=\frac{n-2}{n-1}\theta_0,$$
is an lcb manifold. Indeed,
\begin{align*}
    d\omega^{n-1}&=d\big(\omega_T+\omega_0\big)^{n-1}\\
    &=d\big((n-1)\omega_T\wedge\omega_0^{n-2}\big)\\
    &=(n-2)\theta_0\wedge \big((n-1)\omega_T\wedge \omega_0^{n-2}\big)\\
    &=(n-2)\theta_0\wedge \omega^{n-1}\\
    &=(n-1)\bigg(\frac{n-2}{n-1}\theta_0\bigg)\wedge \omega^{n-1}.
\end{align*}
    However, $(X,\omega)$ is not lcK, since
    $$d\omega=d\omega_0=\theta_0\wedge\omega_0\ne\theta\wedge\omega.$$
    Let $(s,t)$ be isothermal coordinates on $(T,g_T,\omega_T)$ such that $g_T=c(ds^2+dt^2)$. Then the vector field $\partial_t$ is Bismut parallel with respect to the Bismut connection $\nabla^b$ of $g$, while the Levi--Civita covariant derivative of the Lee form $\theta$ satisfies
    $$\nabla^g\theta=\frac{n-2}{n-1}\nabla^{g_0}\theta_0.$$
Therefore, if $(M,g_0,\omega_0,\theta_0)$ is not Vaisman, then
$$\nabla^g\theta\ne 0.$$ 
\end{rem}

\section{The holonomy of the Chern connection}
We now consider the case $t=1$, corresponding to the Chern connection. In this situation, equations
\eqref{ddt3}, \eqref{ddt4}, and \eqref{dJthetatcase} yield
\begin{align*}
    dJ\theta=\big(V(a)+JV(b)\big)V^\flat\wedge JV^\flat.
\end{align*}
Using \eqref{codifft}, we compute
\begin{equation}\label{VaJVbaabb}
    \begin{split}
        V(a)+JV(b)&=-\delta\big(aV^\flat\big)+a\delta V^\flat-\delta\big(bJV^\flat\big)+b\delta JV^\flat\\
        &=-\delta\big(aV^\flat\big)-\delta\big(bJV^\flat\big)+\big(a^2+b^2\big)(1-n).
    \end{split}
\end{equation}
Combining this with \eqref{again}, we obtain
\begin{align*}
    -\delta\big(aV^\flat\big)-\delta\big(bJV^\flat\big)+\big(a^2+b^2\big)(1-n)&=dJ\theta(V,JV)\\
    &=\frac{1}{2}\sum_{i=1}^{2n}dJ\theta(e_i,Je_i)\\
    &=-\delta \theta+(1-n)|\theta|^2.
\end{align*}
Integrating over $M$, we conclude that $\theta_0\equiv 0$. Hence,
\begin{align*}
    \theta&=aV^\flat+bJV^\flat.
\end{align*}
By \eqref{ddt1} and \eqref{ddt2}, we have $dV^\flat=dJV^\flat=0$. Since $d\theta=0$, it follows that
\begin{align}\label{a,b depend on s,t}
    da\wedge V^\flat+db\wedge JV^\flat=0.
\end{align}
Since $t=1$ and $\theta_0\equiv 0$, equation \eqref{tLiederivative}
reduces to
\begin{equation}\label{Liederivativet=1}
\begin{split}
    \mathcal{L}_Vg=a\big(g-V^\flat\otimes V^\flat-JV^\flat\otimes JV^\flat\big),\\
    \mathcal{L}_{JV}g=b\big(g-V^\flat\otimes V^\flat-JV^\flat\otimes JV^\flat\big).
    \end{split}
\end{equation}

We now pass to the universal cover $\tilde{M}$ of $M$, and keep the same notation for the lifted objects. Since $\theta$ is closed, there exists a smooth function $\varphi$ on $\tilde{M}$ such that $\theta=d\varphi$. We decompose the metric $g$ as
\begin{align*}
    g=V^\flat\otimes V^\flat+JV^\flat\otimes JV^\flat+g_0.
\end{align*}
Since $V^\flat$ and $JV^\flat$ are closed, we have 
\begin{equation}\label{liederivativeszero}
    \mathcal{L}_V \big(V^\flat\otimes V^\flat\big)=\mathcal{L}_V \big(JV^\flat\otimes JV^\flat\big)=\mathcal{L}_{JV} \big(V^\flat\otimes V^\flat\big)=\mathcal{L}_{JV} \big(JV^\flat\otimes JV^\flat\big)=0.
\end{equation}
Therefore, \eqref{Liederivativet=1} yields
\begin{equation}\label{Liederivativegzero}
    \begin{split}
        \mathcal{L}_Vg_0&=\mathcal{L}_V g=ag_0,\\
        \mathcal{L}_{JV}g_0&=\mathcal{L}_{JV}g=bg_0.
    \end{split}
\end{equation}
Consequently, since $V(\varphi)=\theta(V)=a$ and $JV(\varphi)=\theta(JV)=b$,
\begin{equation}\label{otherLiederivativeszero}
   \mathcal{L}_V\big(e^{-\varphi}g_0\big)=\mathcal{L}_{JV}\big(e^{-\varphi}g_0\big)=0.
\end{equation}
Consider the metric 
\begin{align*}
    g_1=V^\flat\otimes V^\flat+JV^\flat\otimes JV^\flat+e^{-\varphi}g_0.
\end{align*}
Observe that the metric duals of $V$ and $JV$ are the same with respect to both $g$ and $g_1$.
By \eqref{liederivativeszero} and \eqref{otherLiederivativeszero}, $V$ and $JV$ are Killing vector fields with respect to $g_1$. Moreover, since $V^\flat$ and $JV^\flat$ are closed, $V$ and $JV$ are parallel with respect to the Levi--Civita connection $\nabla^{g_1}$ of $g_1$.
\begin{lem}
    The metric $g_1$ is complete.
\end{lem}
\begin{proof}
    Suppose by contradiction that $(\tilde{M},g_1)$ is not complete. Then there exists a unit-speed maximal geodesic
    $$\gamma : [0,T) \to \tilde{M},$$ 
    with $T \in (0,+\infty)$ and $\gamma(0)=P$. Since $\gamma$ is maximal, for any compact subset $K\subset \tilde{M}$ one has $\gamma\big([0,T)\big)\not\subset K$.
Hence, there exists an exhaustion $(K_n)_n$ of $\tilde{M}$ by compact subsets and a sequence $(t_n)_n\subset[0,T)$ such that $\gamma(t_n)\notin K_n$ and
\begin{align}\label{notinfinity}
    d_{ g_1}\big(P,\gamma(t_n)\big)\not\to +\infty.
\end{align}
Since $a,b$ are defined on the compact manifold $M$, they are bounded on $\tilde{M}$. Thus, $|a|,|b| < \alpha$ for some constant $\alpha$.
We now show that, the function $\varphi$ is bounded along $\gamma$. Indeed,
\begin{align*}
    \bigg|\frac{d}{dt}\varphi\big(\gamma(t)\big)\bigg|&=\big|d\varphi\big(\dot\gamma\big)\big|\\
    &=\big|a\langle V,\dot{\gamma}\rangle_{g_1}+b\langle JV,\dot{\gamma}\rangle_{g_1}\big|\le2\alpha .
\end{align*}
In particular, there exists a constant
$$k:=|\varphi(P)|+2\alpha T $$
such that for any  $t\in[0,T)$,
$$\big|\varphi\big(\gamma(t)\big)\big|\le k.$$
Define the metric
$$g_k:=V^\flat\otimes V^\flat+JV^\flat\otimes JV^\flat+e^{-k}g_0.$$
Then $g_k$ is complete, since it is the pullback of a metric on $M$. Moreover, along $\gamma$ we have
$$g_k\le g_1,$$
which, by the Hopf--Rinow theorem \cite[Theorem 2.8]{dC92}, contradicts (\ref{notinfinity}).
\end{proof}
Applying the De Rham theorem \cite{KN96}, we deduce that the universal cover $\tilde{M}$, endowed with the metric $\tilde{g}_1$, splits isometrically as
\begin{align*}
    \big(\tilde{M},\tilde{g_1}\big)=\big(\mathbb{R}^2\times N,ds^2+dt^2+g_N\big)
\end{align*}
where $(N,g_N,\omega_N)$ is a complete simply connected manifold of complex dimension $n-1$, and $V=\partial_s$, $JV=\partial_t$. We claim that $(N,g_N,\omega_N)$ is K\"ahler. Indeed,
\begin{align*}
    d\omega_N&=d\big(e^{-\varphi}\omega_0\big)\\
    &=-e^{-\varphi}\theta\wedge\omega_0+e^{-\varphi}d\omega_0\\
    &=-e^{-\varphi}\theta\wedge\omega+e^{-\varphi}d\omega=0.
\end{align*}
By \eqref{a,b depend on s,t} and the fact that $\theta=aV^\flat+bJV^\flat$, the function $\varphi$ defined by $d\varphi=\theta$ depends only on $s$ and $t$.
Furthermore, by construction, we obtain the following isometry
\begin{align*}
    \big(\tilde{M},\tilde{g}\big)=\big(\mathbb{R}^2\times N,ds^2+dt^2+e^{\varphi}g_N\big).
\end{align*}
We now show that $\varphi$ descends to $M$ using the following result from \cite[Lemma 3.4]{Mor15}.
\begin{lem}
Assume that  $(N^d,g_N)$ is a complete simply connected Riemannian manifold of dimension $d\ge 1$, $c:\mathbb{R}^2\to\mathbb{R}$ is a smooth function and $\Gamma$ is a co-compact group acting totally discontinuously by isometries on the Riemannian manifold $(\mathbb{R}^2\times N,\
d s^2+d t^2+e^{c(s,t)}g_N)$. Assume moreover that $\Gamma$ preserves the vector fields $\partial_s$ and $\partial_t$. Then the function $c$ is invariant by $\Gamma$.
\end{lem}
We note that in \cite[Lemma 3.4]{Mor15} the function $c$ is assumed to depend only on $t$. However, the argument applies verbatim to the case where $c$ depends on both $s$ and $t$.\\
In particular, $\varphi$ is invariant under the deck transformation group of $g$. Hence $\varphi$ descends to $M$, and consequently $\theta$ is exact on $M$. Therefore, $\omega$ is globally conformally K\"ahler.\\

We conclude this section by providing an example of a compact lcK manifold admitting a nontrivial Chern-parallel vector field and a nonvanishing exact Lee form.

Let $T^2$ be a complex torus endowed with a Hermitian metric $m:=f(x,y)(dx^2+dy^2)$, written in isothermal coordinates, and let $(K,g_K)$ be a compact K\"ahler manifold.
Then $X:=T^2\times K$ is K\"ahler with respect to the product metric 
\begin{align*}
    g&:=f(x,y)\big(dx^2+dy^2\big)+g_K.
\end{align*}
By construction, the conformally related metric $h:=f^{-1}g$ is gcK with Lee form $\theta=d\big(-\text{log}f\big)=-\frac{df}{f}$. Let $W\in TK\subset TX$. Using \cite[Theorem 1.159]{Bes08} and \eqref{nablatt}, we obtain
\begin{align*}
    0&=\nabla^g_W\partial_x\\
    &=\nabla^h_W\partial_x-\frac{1}{2}\big(\theta(W)\partial_x+\theta(\partial_x)W-h(\partial_x,W)\theta^\sharp\big)\\
    &=\nabla^h_W\partial_x-\frac{1}{2}\theta(\partial_x)W\\
    &=\nabla^{1,h}_W\partial_x+\frac{1}{2}\theta(\partial_x)W-\frac{1}{2}\theta(\partial_x)W=\nabla^{1,h}_W\partial_x.
\end{align*}
Moreover, since $m$ is conformally flat we compute
\begin{align*}
\nabla^{1,h}_{\partial_x}\partial_x&=\nabla^h_{\partial_x}\partial_x\\
&=\nabla^g_{\partial_x}\partial_x+\frac{1}{2}\big(\theta(\partial_x)\partial_x+\theta(J\partial_x)J\partial_x-\theta^\sharp\big)\\
&=\nabla^m_{\partial_x}\partial_x+\frac{1}{2}\big(\theta(\partial_x)\partial_x+\theta(J\partial_x)J\partial_x-\theta^\sharp\big)\\
&=\nabla^{flat}_{\partial_x}\partial_x=0,
\end{align*}
and similarly
\begin{align*}
\nabla^{1,h}_{J\partial_x}\partial_x=\nabla^{1,h}_{\partial_x}J\partial_x&=\nabla^{1,h}_{J\partial_x}J\partial_x=0.
\end{align*}
We conclude that $\partial_x,\partial_y=J\partial_x$ are parallel with respect to the Chern connection $\nabla^{1,h}$ of $h$. If $f$ is nonconstant, then $\theta\not\equiv 0$, and $h$ is a non-K\"ahler gcK metric.

\section{\texorpdfstring{The exceptional case $t=\frac{1}{3}$}{t=1/3}}

In this section, we analyze the remaining case $t=\frac{1}{3}$.
We first prove, without any dimension restriction, that $a=b=0$, which means $V$ and $JV$ are orthogonal to the Lee vector field $\theta^\sharp$. After computing some further relations satisfied by the vector fields $V,\ JV$ and the $1-$forms $\theta$ and $J\theta$, we then specialize our arguments to the complex dimension two case. Direct computations, which will be used later in the analysis of the universal cover, allow us to compute the Lie derivative of the metric $g$ along the anti-Lee vector field $J\theta^\sharp$ as well as the brackets involving the vector fields $V,\ JV,\ \theta^\sharp$, and $J\theta^\sharp$.

On the universal cover $\tilde M$, we write the Lee form $\theta$ as $\theta=d\varphi$ and identify $\tilde M$ with $S\times \mathbb R^2$, where $S$ is a connected complete surface, using the flows of suitable commuting vector fields. We then note that the vector field $e^{\frac23 \varphi}J\theta^\sharp$ is Killing on $\tilde M$ and determines a Killing vector field on $S$ whose zero set is discrete and non-empty. Therefore, $S$ is a surface of revolution, diffeomorphic either to $\mathbb S^2$ or to $\mathbb R^2$. Both cases imply that $\varphi$ is invariant under the action of the deck transformation group of $\tilde M$. Hence, the metric is gcK. \\

We start by proving our claim that $a=b=0$.
Wedging a $1$-form with $\omega$ is injective. Therefore, from (\ref{ddt4}) we obtain
\begin{align}\label{da in the 1over3 case}
    da=-\frac{2}{3}a\theta, \qquad db=-\frac{2}{3}b\theta.
\end{align}
In particular,
$$V(a)=-\frac{2}{3}a^2,\qquad JV(b)=-\frac{2}{3}b^2.$$ 
Using (\ref{codifft}), we compute
\begin{align*}
    \delta\big(aV^\flat)&=-V(a)+a\delta V^\flat\\
    &=\frac{2}{3}a^2+\frac{2}{3}(1-n)a^2=\frac{2}{3}(2-n)a^2,
    \end{align*}
    and similarly,
    \begin{align*}
    \delta(bJV^\flat)&=\frac{2}{3}(2-n)b^2.
\end{align*}
If $n>2$, integrating the above identities over $M$, we conclude directly that $a=b=0$. On the other hand, if $n=2$, then
$$\delta(aV^\flat)=\delta(bJV^\flat)=0,$$
which implies
\begin{align}\label{deltathetadeltathetazero case 1over3}
    \delta\theta=\delta\theta_0.
\end{align}
Therefore,
\begin{align*}
    -\frac{2}{3}\int a^2|\theta|^2&\overset{\eqref{da in the 1over3 case}}{=}\frac{1}{2}\int \theta^\sharp\big(a^2\big)=\frac{1}{2}\int a^2 \delta \theta \\
    &\overset{\eqref{deltathetadeltathetazero case 1over3}}{=}\frac{1}{2}\int a^2 \delta\theta_0=\frac{1}{2}\int \theta_0^\sharp\big(a^2\big)\\
    &\overset{\eqref{da in the 1over3 case}}{=}-\frac{2}{3}\int a^2|\theta_0|^2.
\end{align*}
We deduce that
$$\int a^4+a^2b^2=0,$$
and hence $a=0$. By the same argument, we also obtain $b=0$, thus proving our claim.\\

Identities \eqref{ddt1} and \eqref{ddt2} become
\begin{equation}\label{differential1/3}
    \begin{split}
        dV^\flat&=\frac{1}{3}\theta\wedge V^\flat,\\
        dJV^\flat&=\frac{1}{3}\theta\wedge JV^\flat.
    \end{split}
\end{equation}
For $t=\frac{1}{3}$, the Levi--Civita covariant derivatives of $V$ and $JV$ are given by \eqref{parallel}:
\begin{equation}\label{nabla1/3}
    \begin{split}
        \nabla^g_X V&=-\frac{1}{6}J\theta(X)JV-\frac{1}{6}JV^\flat(X)J\theta^\sharp-\frac{1}{3}V^\flat(X)\theta^\sharp,\\
        \nabla^g_X JV&=\frac{1}{6}J\theta(X)V+\frac{1}{6}V^\flat(X)J\theta^\sharp-\frac{1}{3}JV^\flat(X)\theta^\sharp.
    \end{split}
\end{equation}
By \eqref{tLiederivative}, the Lie derivative of $g$ along $V$ is
\begin{align}\label{62}
    \mathcal{L}_V g=-\frac{1}{3}\big(\theta\cdot V^\flat+J\theta\cdot JV^\flat\big).
\end{align}
Since $\theta(V)=a=0$ and $\theta=g(\theta^\sharp,\cdot)$, we have
\begin{align*}
    0&=d\iota_V\theta+\iota_V d\theta=\mathcal{L}_V\theta\\
    &=\mathcal{L}_V g\big(\theta^\sharp\big)+g\big(\mathcal{L}_V\theta^\sharp\big)\\
    &=-\frac{1}{3}|\theta|^2V^\flat+g\big(\mathcal{L}_V\theta^\sharp\big).
\end{align*}
It follows that
\begin{align}\label{LieVthetasharp1over3}
    \mathcal{L}_V\theta^\sharp=\frac{1}{3}|\theta|^2 V.
\end{align}
Consequently,
\begin{align*}
    V\big(|\theta|^2\big)&=V\big(\theta(\theta^\sharp)\big)\\
&=\mathcal{L}_V\theta\big(\theta^\sharp\big)+\theta\big(\mathcal{L}_V\theta^\sharp\big)=0.
\end{align*}
Applying the same argument to $JV$, we obtain
\begin{align}\label{Vnormthteaauqred}
    V\big(|\theta|^2\big)=JV\big(|\theta|^2\big)=0,
\end{align}
and
\begin{align}\label{3-2nLieJVtheta}
    \mathcal{L}_{JV} \theta^\sharp=\frac{1}{3}|\theta|^2JV.
\end{align}
    Since, by \eqref{brackett},
\begin{align}\label{brackett1}
    [V,JV]=\frac{1}{3}J\theta^\sharp,
\end{align}
the Jacobi identity together with \eqref{Vnormthteaauqred} implies
\begin{equation}\label{60}
    J\theta^\sharp\big(|\theta|^2\big)=0,
\end{equation}
and from \eqref{LieVthetasharp1over3} and \eqref{3-2nLieJVtheta} we obtain:
\begin{equation}
    \begin{split}
   \mathcal{L}_{J\theta^\sharp}\theta^\sharp&=3\mathcal{L}_{[V,JV]}\theta^\sharp\\
    &=3\big(\mathcal{L}_V\mathcal{L}_{JV}\theta^\sharp-\mathcal{L}_{JV}\mathcal{L}_V\theta^\sharp\big)\\
    &\overset{\eqref{LieVthetasharp1over3},\eqref{Vnormthteaauqred},\eqref{3-2nLieJVtheta}}{=}|\theta|^2[V,JV]=\frac{2}{3}|\theta|^2J\theta^\sharp.
    \end{split}
\end{equation}
Therefore,
    \begin{equation}\label{dJtheta theta Jtheta 1over3}
    \begin{split}
dJ\theta\big(\theta^\sharp,J\theta^\sharp\big)&=\theta^\sharp\big(J\theta(J\theta^\sharp)\big)-J\theta^\sharp\big(J\theta(\theta^\sharp)\big)-J\theta\big([\theta^\sharp,J\theta^\sharp]\big)\\
&=\theta^\sharp\big(|\theta|^2\big)+\frac{2}{3}|\theta|^4.
  \end{split}
\end{equation}
Since $\theta^\sharp$ is orthogonal to both $V,JV$, using \eqref{LieVthetasharp1over3}, we obtain
\begin{equation}\label{dJt1}
\begin{split}
    dJ\theta(V,\theta^\sharp)&=V\big(J\theta(\theta^\sharp)\big)-\theta^\sharp\big(J\theta(V)\big)-J\theta\big([V,\theta^\sharp]\big)\\
    &\overset{\eqref{LieVthetasharp1over3}}{=}-\frac{1}{3}J\theta\big(|\theta|^2V\big)=0.
\end{split}
\end{equation}
Similarly, \eqref{3-2nLieJVtheta} yields
\begin{equation}\label{dJt2}
    dJ\theta(JV,\theta^\sharp)=0.
\end{equation}
Since $dJ\theta$ is a $(1,1)$-form, it is $J$-invariant. Hence, using \eqref{dJt1} and \eqref{dJt2} we obtain
\begin{equation}\label{dJt3}    dJ\theta(JV,J\theta^\sharp)=dJ\theta(V,J\theta^\sharp)=0.
\end{equation}
Finally, \eqref{dJthetatcase} gives
\begin{equation}\label{dJt4}
    dJ\theta(V,JV)=-J\theta([V,JV])=-\frac13|\theta|^2.
\end{equation}

Suppose now that $n=2$. On the open set $T:=\{\theta\neq 0\}$ we can consider the orthonormal frame given by 
$$V,\ JV,\ \frac{\theta^\sharp}{|\theta|},\ \frac{J\theta^\sharp}{|\theta|}.$$
Using this frame, \eqref{again} together with \eqref{dJthetatcase} and \eqref{dJtheta theta Jtheta 1over3} becomes
\begin{align*}
    \delta\theta+\frac{\theta^\sharp\big(|\theta|^2\big)}{|\theta|^2}+\frac{4}{3}|\theta|^2=0.
\end{align*}
Consequently, on $T$ we have
\begin{align}\label{delta}
    \delta\bigg(\frac{\theta}{|\theta|^2}\bigg)+\frac{4}{3}=0.
\end{align}
From the Cartan formula we easily obtain 
\begin{equation}\label{69}
    \mathcal{L}_{J\theta^\sharp}\theta= \mathcal{L}_{J\theta^\sharp}V^\flat= \mathcal{L}_{J\theta^\sharp}JV^\flat=0.
\end{equation}
Moreover, from \eqref{dJtheta theta Jtheta 1over3}--\eqref{dJt4} we obtain at each point of $T$:
\begin{equation}\label{72}
    dJ\theta=-\frac13|\theta|^2V^\flat\wedge JV^\flat+\left(\frac23+\frac{\theta^\sharp(|\theta|^2)}{|\theta|^4}\right)\theta\wedge J\theta.
\end{equation}
Using \eqref{Vnormthteaauqred} and \eqref{60} we get at each point of $T$ 
\begin{equation}\label{73}
    d|\theta|^2=\frac{\theta^\sharp(|\theta|^2)}{|\theta|^2},
\end{equation}
so using \eqref{72}, \eqref{73}, and the Cartan formula again:
\begin{equation}\label{74}
    \mathcal{L}_{J\theta^\sharp}J\theta=d|\theta|^2-|\theta|^2\left(\frac23+\frac{\theta^\sharp(|\theta|^2)}{|\theta|^4}\right)\theta=-\frac23|\theta|^2\theta.
\end{equation}
On $T$ one can write the metric $g$ as
\begin{equation}\label{75}
    g=V^\flat\otimes V^\flat+JV^\flat\otimes JV^\flat+\frac1{|\theta|^2}(\theta\otimes\theta+J\theta\otimes J\theta),
\end{equation}
so from \eqref{69} and \eqref{74}, we obtain
\begin{equation} \label{76}
    \mathcal{L}_{J\theta^\sharp}g=-\frac23\theta\cdot J\theta.
\end{equation}
This equation also holds on the interior of $M\setminus T$ (where $\theta=0$), so by density it holds on $M$.
From \eqref{LiederivativeofJ} with $t=1/3$ we get
\begin{align}
        \mathcal{L}_V J=\frac{1}{3}\big(JV\wedge \theta+V\wedge J\theta\big),
    \end{align}
    so together with \eqref{LieVthetasharp1over3} we obtain 
    \begin{equation}\label{78}
        \mathcal{L}_V (J\theta^\sharp)= \mathcal{L}_{JV} (J\theta^\sharp)=0.
    \end{equation}

Consider now the universal cover $\tilde M$ of $M$ and denote the pull-backs with the same symbols for simplicity. We write $\theta=d\varphi$ for some $\varphi\in C^\infty(\tilde M)$. An easy calculation shows that 
$$\mathcal{L}_{(fX)}g=f\mathcal{L}_{X}g+df\cdot X^\flat$$ for every vector field $X$ and function $f$, so from \eqref{76} we obtain that $\xi:=e^{\frac23\varphi}J\theta^\sharp$ is a Killing vector field on $\tilde M$.  

Now, from \eqref{differential1/3} one can write $e^{-\frac13\varphi}V^\flat=ds$ and $e^{-\frac13\varphi}JV^\flat=dt$ for some functions $s,t$ on $M$, whose differentials never vanish. Denoting by $S$ some level set of the function $(s,t):\tilde M\to\mathbb{R}^2$, one can identify $\tilde M$ with $S\times \mathbb{R}^2$ via the flows of commuting vector fields dual to $ds$ and $dt$, like for instance $\partial_s:=e^{\frac13\varphi}V$ and $\partial_t:=e^{\frac13\varphi}JV-\frac{s}3e^{\frac23\varphi}J\theta^\sharp$ (which commute thanks to \eqref{brackett1} and \eqref{78}). In particular $S$ is connected.

The Killing vector field $\xi$ is tangent to $S$ at each point, and its restriction to $S$ is again Killing. Let $Z$ denote its zero set on $S$, which is a totally geodesic submanifold of even codimension. Since $\xi$ commutes with $\partial_s$ and $\partial_t$ by \eqref{78}, its zero set on $\tilde M$ (which is the set $\{\theta=0\}$) is equal to $Z\times\mathbb{R}^2$. On the other hand, $\theta$ is assumed to be non-identically zero, so $Z$ is not equal to $S$. Thus $Z$ is discrete in $S$.

Since moreover $Z$ is non-empty (otherwise we obtain a contradiction from \eqref{delta}), it follows that $S$ is a surface of revolution. It is thus either diffeomorphic to $S^2$ if $\xi$ has two zeros on S, or to $\mathbb{R}^2$  if $\xi$ has only one zero on $S$.

In the first case, $S$ is compact, so the function $\varphi$, which is constant in the direction of $\partial_s$ and $\partial_t$, is bounded on $\tilde M$. On the other hand, every $\gamma\in\Gamma:=\pi_1(M)$ preserves $\theta$, so satisfies $\gamma^*\varphi=\varphi+c_\gamma$ for some constant $c_\gamma$. Since $\varphi$ is bounded, one obtains $c_\gamma=0$ for every $\gamma\in\pi_1(M)$, so $\varphi$ is the pull-back of a function on $M$, whence $M$ is gcK. Moreover, every $\gamma\in\pi_1(M)$ satisfies $\gamma_*\xi=\xi,\gamma_*\partial_s=\partial_s,$ and preserves $ds$ and $dt$, so one can write for every $(x,s,t)\in S\times \mathbb{R}^2$:
\begin{equation}
    \gamma(x,s,t)=(\gamma_0(x,t),s+s_0,t+t_0),
\end{equation}
for some constants $s_0,t_0$ and diffeomorphisms $\gamma_0(\cdot,t)$ of $S$ preserving $\xi$.

In the second case, if $x_0$ denotes the unique zero of $\xi$ on $S$, one can check that for every unitary geodesic $c$ on $S$ with $c(0)=x_0$, the tangent vector $\dot{c}$ is proportional to $\theta^\sharp$ along $c$, with different signs when the time parameter changes sign, that is, there exists $\varepsilon\in\{\pm1\}$ such that $|\theta|\dot{c}(t)=\varepsilon\theta$ for $t\ge 0$ and $|\theta|\dot{c}(t)=-\varepsilon\theta$ for $t\le 0$. Therefore, the function $t\mapsto \varphi(c(t))$ has an absolute minimum at $0$ if $\varepsilon=1$ and an absolute maximum at $0$ if $\varepsilon=-1$, which shows that $\varphi$ is bounded on $\tilde M$ either from below if $\varepsilon=1$, or from above if $\varepsilon=-1$. As before, this implies that $\varphi$ is not surjective, so $M$ is gcK. \\

In the next example, following the above arguments, we construct a compact complex surface
admitting a nonzero $\nabla^{1/3}$-parallel vector field with respect to a
gcK metric that is not K\"ahler.

\begin{ex}\label{examplegcK_1/3}
Let $\tilde M:=\mathbb S^2\times\mathbb R^2$, with angular coordinates
$(\alpha,\beta)$ on $\mathbb S^2$ away from the poles and coordinates $(s,t)$
on $\mathbb R^2$. 
Extending the complex structure on $\mathbb{S}^2$, we define an almost complex structure $J$ on $\tilde{M}$ by
$$
J\partial_\alpha:=-\frac{1}{\sin\alpha}\partial_\beta,
\qquad
J\partial_s:=\partial_t+\frac{s}{3}\partial_\beta.
$$
Since
$$[\partial_\alpha-iJ\partial_\alpha,\partial_s-iJ\partial_s]=0,$$
$J$ is a complex structure on $\tilde{M}$. We recall that the vector field $\partial_\beta$ is Killing for the standard metric
$$
g_{0,\mathbb S}:=d\alpha^2+\sin^2\alpha\,d\beta^2
$$
on $\mathbb{S}^2$. Since $\cos\alpha$ is independent of $\beta$, $\partial_\beta$ is also Killing for
\begin{align}\label{metric_on_S^2}
    g_{\mathbb S}:=e^{\frac23\cos\alpha}
\bigl(d\alpha^2+\sin^2\alpha\,d\beta^2\bigr).
\end{align}
Set $\varphi:=\cos\alpha$ and $\theta:=d\varphi$, and define
$$
V:=e^{-\varphi/3}\partial_s,
\qquad \textrm{so that }
JV=e^{-\varphi/3}
\left(\partial_t+\frac{s}{3}\partial_\beta\right).
$$
We note that the vector field $\partial_\beta$ is also Killing for the Riemannian metric
$$
\tilde g
:=e^{\frac23\varphi}
\left(
d\alpha^2+ds^2+dt^2+
\sin^2\alpha\left(d\beta-\frac{s}{3}dt\right)^2
\right)
$$
on $\tilde{M}$. Since in the standard coordinates $(x,y,z)$ of $\mathbb{R}^3$ we have
\begin{align*}
    \sin^2\alpha \,d\beta=xdy-ydx,\quad  \quad \sin^2\alpha=x^2+y^2,
\end{align*}
and the metric \eqref{metric_on_S^2} is smooth, $\tilde g$ extends smoothly across the poles.
The metric $\tilde g$ is Hermitian with respect to $J$ and the fundamental 2-form $\tilde{\omega}$ associated to $(J,\tilde g)$ is given by
$$\tilde\omega
=e^{\frac23\varphi}
\left(-\sin\alpha\,d\alpha\wedge\left(d\beta-\frac{s}{3}dt\right)+ds\wedge dt\right),$$
and satisfies
$d\tilde\omega=\theta\wedge\tilde\omega$. Thus $\tilde g$ is
gcK and $\theta$ is the Lee form of $\tilde{\omega}$. Moreover,
$$
\theta^\sharp
=-\sin\alpha\,e^{-\frac23\varphi}\partial_\alpha,
\qquad
J\theta^\sharp=e^{-\frac23\varphi}\partial_\beta.
$$

A direct computation using the Koszul formula and the lcK expression
for the Gauduchon connections \eqref{nablatt} shows that
$V$ is parallel with respect to the Gauduchon connection
$\nabla^{1/3}_{\tilde g}$. To show this, we consider the orthonormal frame $\{e_i\}_{i=1}^4$ given by
\begin{equation*}
    \begin{aligned}
        e_1&=e^{-u}\partial_\alpha,&\quad e_2&=e^{-u}\partial_s=V,\\ 
    e_3&=e^{-u}\left(\partial_t+\frac{s}{3}\partial_\beta\right)=JV,&\quad e_4&=\frac{e^{-u}}{\text{sin}\alpha}\partial_\beta,
    \end{aligned}
\end{equation*}
where $u:=\frac{1}{3}\text{cos}\alpha=\frac{1}{3}\varphi$. 
Setting $\lambda:=\frac{\text{sin}\alpha}{3}e^{-u}$, we obtain
\begin{equation}\label{bracketse_i}
    \begin{aligned}
         [e_1,e_2]&=\lambda e_2,&\quad [e_1,e_3]&=\lambda e_3,\\
    [e_1,e_4]&=e^{-u}\bigg(\frac{\text{sin}\alpha}{3}-\frac{\text{cos}\alpha}{\text{sin}\alpha}\bigg)e_4,&\quad[e_2,e_3]&=\lambda e_4,\\
     [e_2,e_4]&=[e_3,e_4]=0.
    \end{aligned}
\end{equation}
Combining \eqref{bracketse_i} with 
$$\theta^\sharp=-e^{-2u}\text{sin}\alpha \,e_1,\quad J\theta^\sharp=e^{-2u}\text{sin}\alpha\,e_4, $$
and the Koszul formula for orthogonal vector fields $X,Y,Z$,
$$2\tilde{g}(\nabla^{\tilde{g}}_XY,Z)=\tilde{g}\big([X,Y],Z\big)-\tilde{g}\big([Y,Z],X\big)+\tilde{g}\big([Z,X],Y\big),$$
we compute the Levi--Civita covariant derivative of $V$ along $V,JV,\theta^\sharp,J\theta^\sharp$:
\begin{equation*}
    \begin{aligned}
        \nabla_V^{\tilde{g}} V&=-\frac{1}{3}\theta^\sharp,&\quad \nabla_{JV}^{\tilde{g}}V&=-\frac{1}{6}J\theta^\sharp,\\
        \nabla_{\theta^\sharp}^{\tilde{g}}V&=0,&\quad \nabla_{J\theta^\sharp}^{\tilde{g}}V&=-\frac{1}{6}|\theta|^2JV.
    \end{aligned}
\end{equation*}
Hence, for every vector field $X$ on $\tilde{M}$, the Levi--Civita covariant derivative of $V$ in given by
\begin{align}\label{LCforVinthegcKexample}
    \nabla_X^{\tilde{g}}V&=-\frac{1}{3}\tilde{g}(X,V)\theta^\sharp-\frac{1}{6}g(X,JV)J\theta^\sharp-\frac{1}{6}g(X,J\theta^\sharp)JV.
\end{align}
Using \eqref{nablatt} and \eqref{LCforVinthegcKexample}, we obtain that $V$ is parallel with respect to the Gauduchon connection
$\nabla^{1/3}_{\tilde g}$.

Let $a,b>0$ be positive numbers such that $ab/3\in2\pi\mathbb Z$, and consider the free and properly discontinuous actions
\begin{align*}
\gamma_a(\alpha,\beta,s,t)
  &:=\left(\alpha,\beta+\frac{a}{3}t,s+a,t\right),\\
\gamma_b(\alpha,\beta,s,t)
  &:=(\alpha,\beta,s,t+b).
\end{align*}
Direct computations show that $\gamma_a,\gamma_b$ commute and preserve $J$,$\tilde g$, $\theta$, and $V$.
Consequently, the compact complex surface
$$
M:=\tilde M/\langle \gamma_a,\gamma_b\rangle
$$
inherits a gcK metric $g$ and a nonzero $\nabla^{1/3}_g$-parallel vector
field induced respectively by $\tilde g$ and $V$. The metric $g$ is not K\"ahler, since its Lee form $\theta$
is not identically zero.
\end{ex}
\begin{rem}\label{high_dim_ex}
Example~\ref{examplegcK_1/3} provides a compact complex surface
$(M,J,g_M,\omega_M,\theta=d\varphi)$ endowed with a non-K\"ahler gcK metric $g_M$
and a nontrivial vector field $V$ that is parallel with respect to the
Gauduchon connection $\nabla^{1/3}_{g_M}$. This yields higher-dimensional examples. Indeed, for any
compact K\"ahler manifold $(X,g_X)$, the product $P:=M\times X$, equipped
with the metric $g:=g_M\oplus e^{\varphi}g_X$, is a compact Hermitian
manifold of complex dimension $2+\dim_{\mathbb C}X$ admitting the non-K\"ahler gcK
metric $g$ such that
the extension of $V$ to $P$ is a nontrivial
$\nabla^{1/3}_g$-parallel vector field.
This follows from the conformal change formula for the Levi--Civita connection \cite[Theorem 1.159]{Bes08}, equation \eqref{nablatt}, and $\theta(V)=\theta(JV)=0$. Indeed, for any vector field $Y$ on $P$, we have
\begin{equation*}
    \begin{split}
        {\nabla^{\frac13}_g} _Y V&=\nabla^{g_M+e^\varphi g_X}_Y V-\frac{1}{6}\theta(JY)JV-\frac16 g(JY,V)J\theta^\sharp+\frac{1}{3}g(Y,V)\theta^\sharp\\
        &=\nabla^{e^\varphi(e^{-\varphi}g_M+g_X)}_YV-\frac{1}{6}\theta(JY)JV-\frac16g_M(JY,V)J\theta^\sharp+\frac13g_M(Y,V)\theta^\sharp\\
        &=\nabla^{e^{-\varphi}g_M}_Y V+\nabla^{g_X}_Y V+\frac12\theta(Y)V+\frac12\theta(V)Y-\frac12 g_M(Y,V)\theta^\sharp\\
        &-\frac{1}{6}\theta(JY)JV-\frac{1}{6}g_M(JY,V)J\theta^\sharp+\frac13g_M(Y,V)\theta^\sharp\\
        &=\nabla^{g_M}_YV-\frac{1}{6}\theta(JY)JV-\frac{1}{6}g_M(JY,V)J\theta^\sharp+\frac13g_M(Y,V)\theta^\sharp={\nabla^{\frac13}_{g_M}} _Y V=0.
    \end{split}
\end{equation*}
\end{rem}

\bibliographystyle{alpha}

\end{document}